\documentclass[a4paper,final]{siamltex}

\usepackage{amsmath,amsfonts,amssymb}
\usepackage[all]{xy}      
\usepackage{mathrsfs} 
\usepackage{url}
\usepackage{color}

\DeclareMathOperator{\spann}{span}
\DeclareMathOperator{\im}{im} 
\DeclareMathOperator{\inte}{int}
\DeclareMathOperator{\sign}{sign}

\DeclareMathOperator{\add}{add} 
\DeclareMathOperator{\e}{e} 
\DeclareMathOperator{\bd}{d} 

\newcommand{\RR}{{\mathbb R}}
\newcommand{\NN}{{\mathbb N}}
\newcommand{\expp}[1]{\e^{\, #1}}
\newcommand{\mal}{\circ}
\newcommand{\st}{\, | \,}
\newcommand{\dd}[2]{\frac{d #1}{d #2}}
\newcommand{\DD}[2]{\frac{\partial #1}{\partial #2}}

\newcommand{\kin}{kinetic }
\newcommand{\kinsub}{kinetic-order }

\title{Generalized mass action systems:\\
Complex balancing equilibria and sign vectors
of the stoichiometric and kinetic-order subspaces}

\author{Stefan M\"uller\thanks{Johann Radon Institute for Computational and Applied Mathematics,
Austrian Academy of Sciences, Altenberger Stra{\ss}e 69, 4040 Linz, Austria ({\tt stefan.mueller@ricam.oeaw.ac.at}).}
\and Georg Regensburger\thanks{INRIA Saclay -- \^{I}le de France, Project DISCO, L2S, Sup\'{e}lec,
3 rue Joliot-Curie, 91192 Gif-sur-Yvette Cedex, France ({\tt georg.regensburger@ricam.oeaw.ac.at}).
Supported by the Austrian Science Fund (FWF): J3030-N18.}}

\begin{document}

\maketitle

\begin{abstract}
Mass action systems capture chemical reaction networks in homogeneous and dilute solutions.
We suggest a notion of generalized mass action systems
that admits arbitrary 
power-law rate functions
and serves as a more realistic model for reaction networks in intracellular environments.
In addition to the complexes of a network and the related stoichiometric subspace,
we introduce corresponding \kin complexes,
which represent the 
exponents in the rate functions
and determine the \kinsub subspace.
We show that several results of Chemical Reaction Network Theory
carry over to the case of generalized mass action kinetics.
Our main result essentially states that,
if the sign vectors of the stoichiometric and \kinsub subspace coincide,
there exists a unique complex balancing equilibrium in every stoichiometric compatibility class.
However, in contrast to classical mass action systems,
multiple complex balancing equilibria in one stoichiometric compatibility class are possible in general.
\end{abstract}

\begin{keywords}
chemical reaction network theory, generalized mass action kinetics, complex balancing,
generalized Birch's theorem, oriented matroids
\end{keywords}

\begin{AMS}
92C42, 37C25, 52C40
\end{AMS}

%
%
%

\pagestyle{myheadings}
\thispagestyle{plain}
\markboth{STEFAN M\"ULLER AND GEORG REGENSBURGER}{GENERALIZED MASS ACTION SYSTEMS}


\section{Introduction}

Dynamical systems arising from chemical reaction networks with mass action kinetics
are the subject of Chemical Reaction Network Theory (CRNT),
which was initiated by the work of Horn, Jackson, and Feinberg, cf.~\cite{HornJackson1972,Horn1972,Feinberg1972}.
In particular, this theory provides results about existence, uniqueness, and stability of equilibria
{\em independently} of rate constants (and initial conditions).
However, the validity of the underlying mass action law is limited;
it only holds for elementary reactions in homogeneous and dilute solutions.
In intracellular environments, which are highly structured
and characterized by macromolecular crowding,
the rate law has to be modified, cf.\ \cite{Clegg1984,Halling1989,Kuthan2001}.

Two types of modifications have been proposed:
``fractal reaction kinetics'' \cite{Kopelman1986,Kopelman1988,SchnellTurner2004,GrimaSchnell2006}
and the ``power-law formalism'' \cite{Savageau1969,Savageau1976,Savageau1992,Savageau1995}.
The names of the two approaches are a bit misleading,
since both approaches address the problem of dimensional restriction
(i.e.\ molecules confined to surfaces, channels, or fractal-like structures)
and both use power-laws.
More specifically, in fractal-like kinetics,
rate constants are time-dependent (via a power-law),
whereas the exponents of the species concentrations in the rate function
are the corresponding stoichiometric coefficients (as in mass action kinetics).
On the other hand,
in the power-law formalism,
rate constants are time-independent (as in mass action kinetics),
whereas the exponents of the species concentrations may be (nonnegative) real numbers
different from the respective stoichiometric coefficients.
For model selection, data have to be collected for many molecules and intracellular environments.
Recent data of binding kinetics in crowded media \cite{Bajzer2008,Neff2011}
suggest that the power-law formalism is the preferred model.

In this work,
we study the consequences of the power-law formalism
for chemical reaction networks.
In particular,
we demonstrate that several fundamental results of CRNT
carry over to the case of generalized mass action kinetics
(i.e.\ power-law rate functions).
There has been an early approach to account for generalized mass action kinetics \cite{HornJackson1972},
which entails a redefinition of the complexes of a network.
Here, we suggest a different approach, where we keep the original complexes,
but introduce additional ``\kin complexes'',
which determine the exponents of the species concentrations in the rate functions.
This has the advantage that
the underlying chemical reaction network
and thus properties like weak reversibility and deficiency remain the same.

From the \kin complexes, we obtain (in addition to the stoichiometric subspace) a ``\kinsub subspace'',
and it turns out that the generalization of a central result of CRNT
(concerned with the uniqueness and existence of a complex balancing equilibrium
in every stoichiometric compatibility class)
depends on the sign vectors of the two subspaces.
Our main result Theorem~\ref{genthm} essentially states that,
if these sign vectors are equal,
there exists a unique complex balancing equilibrium in every stoichiometric compatibility class. 
In general, however, there may be more than one complex balancing equilibrium in a stoichiometric compatibility class,
see Proposition~\ref{notunique} and Example~\ref{ex2}.

Chemical reaction networks with non mass action kinetics are also studied in
\cite{BanajiDonnellBaigent2007,BanajiCraciun2010,BanajiCraciun2009}.
In this approach, one is interested in conditions
that guarantee the uniqueness of equilibria.
If autocatalytic reactions are excluded
and if the dependence of the rate functions on the species concentrations corresponds to the stoichiometric matrix,
the structure of the stoichiometric matrix alone guarantees uniqueness.
Moreover, the properties of the stoichiometric matrix
can be translated into conditions for the species reaction graph.
As a consequence, this theory is applicable to many types of kinetics,
however, it does not address the existence of equilibria.
Existence and uniqueness of equilibria for general kinetics are discussed in~\cite{CraciunHeltonWilliams2008}.
The methods are based on homotopy invariance of the Brouwer degree in a related way to the approach in Section~\ref{subsec:surj}.

\subsection*{Organization of the work}

In the next section,
we recall the definition of mass action systems and several fundamental results of CRNT.
Then we introduce generalized mass action systems
and discuss the results that carry over easily to this framework.
In Section \ref{sec:cbe},
we study uniqueness and existence of complex balancing equilibria;
more specifically, we reformulate the problem
and study injectivity and surjectivity of a certain map,
(a simplified version of) which appears for example in toric and computational geometry or statistics.
In Section \ref{sec:exa},
we discuss two examples of generalized mass action systems.
Finally,
we draw our conclusions
and give an outlook to further lines of research.
In the Appendix,
we recall the relevant results on sign vectors of vector spaces and face lattices of polyhedral cones and polytopes.

\subsection*{Notation}
We denote the positive real numbers by $\RR_>$
and the nonnegative real numbers by $\RR_\ge$.
For a finite index set $I$,
we write $\RR^I$ for the real vector space of formal sums $x=\sum_{i \in I} x_i \, i$ with $x_i \in \RR$, and
$\RR^I_>$ and $\RR^I_\ge$ for the corresponding subsets.
Given $x \in \RR^I$,
we write $x>0$ if $x \in \RR^I_>$ and $x \ge 0$ if $x \in \RR^I_\ge$.
Further, we define $\e^x \in \RR^I_>$ and $\ln(x) \in \RR^I$ componentwise,
i.e.\ $(\e^x)_i = \e^{x_i}$ and $(\ln(x))_i = \ln(x_i)$, the latter for $x \in \RR^I_>$.
Finally, we define $x \mal y \in \RR^I$ for $x,y \in \RR^I$ as $(x \mal y)_i = x_i y_i$
and $x^y \in \RR_\ge$ for $x,y \in \RR^I_\ge$ as $x^y = \prod_{i \in I} x_i^{y_i}$,
where we set $0^0=1$.
\newpage
\section{Chemical reaction networks}

In our presentation of CRNT,
we follow the surveys by Feinberg \cite{Feinberg1979,Feinberg1987,Feinberg1995a}
and Gunawardena \cite{Gunawardena}.

\begin{definition} \label{CRN}
A chemical reaction network $(\mathscr{S},\mathscr{C},\mathscr{R})$
consists of three finite sets:
(i) a set $\mathscr{S}$ of species,
(ii) a set $\mathscr{C} \subset \RR_\ge^\mathscr{S}$ of complexes,
and (iii) a set $\mathscr{R} \subset \mathscr{C} \times \mathscr{C}$ of reactions
with the following properties:
(a) $\forall y \in \mathscr{C} \colon \exists y' \in \mathscr{C}$ such that
$(y,y') \in \mathscr{R}$ or $(y',y) \in \mathscr{R}$
and (b) $\forall y \in \mathscr{C} \colon (y,y) \not\in \mathscr{R}$.
\end{definition}

Complexes are formal sums of species; they are the left-hand sides and right-hand sides of chemical reactions.
For $y \in \mathscr{C}$, we may write $y=\sum_{s \in \mathscr{S}} y_s \, s$,
where $y_s$ is the stoichiometric coefficient of species $s$.
As usual in chemistry, we write $y \to y'$ for a reaction $(y,y') \in \mathscr{R}$.
In a chemical reaction network, each complex appears in at least one reaction;
moreover, there are no reactions of the form $y \to y$.

A chemical reaction network $(\mathscr{S},\mathscr{C},\mathscr{R})$ gives rise to a directed graph
with complexes as nodes and reactions as edges.
Connected components $L_1,\ldots,L_l \subseteq \mathscr{C}$
are called {\em linkage classes},
strongly connected components are called {\em strong linkage classes},
and strongly connected components without outgoing edges $T_1,\ldots,T_t \subseteq \mathscr{C}$
are called {\em terminal strong linkage classes}.
Each linkage class must contain at least one terminal strong linkage classes, i.e. $t \ge l$.
The network $(\mathscr{S},\mathscr{C},\mathscr{R})$
is called {\em weakly reversible},
if the linkage classes coincide with the strong linkage classes
and hence with the terminal strong linkage classes.

From a dynamic point of view,
each reaction $y \to y' \in \mathscr{R}$
causes a change in species concentrations proportional to $y' - y \in \RR^\mathscr{S}$.
The change caused by all reactions lies in a subspace of $\RR^\mathscr{S}$
such that any trajectory in $\RR^\mathscr{S}_\ge$ lies in a coset of this subspace.
\begin{definition} \label{SSS}
Let $(\mathscr{S},\mathscr{C},\mathscr{R})$ be a chemical reaction network.
The stoichiometric subspace is defined as
\begin{equation*}
S = \spann \{ y'-y \in \RR^\mathscr{S} \st y \to y' \in \mathscr{R} \} \, .
\end{equation*}
Further, let $c'\in \RR^\mathscr{S}_>$.
The corresponding stoichiometric compatibility class is defined as
\begin{equation*}
(c'+S)_\geq = (c'+S) \cap \RR^\mathscr{S}_\ge \, .
\end{equation*}
\end{definition}


\subsection{Mass action systems}

The rate of a reaction $y \to y' \in \mathscr{R}$
depends on the concentrations of the species involved.
The explicit form of the rate function $\mathscr{K}_{y \to y'} \colon \RR^\mathscr{S}_\ge \to \RR_\ge$
is determined by the underlying kinetics.
In the case of mass action kinetics,
it is a monomial in the concentrations $c \in \RR^\mathscr{S}_\ge$ of reactant species,
i.e.\ $\mathscr{K}_{y \to y'}(c) = k_{y \to y'} \, c^y$
with rate constant $k_{y \to y'} \in \RR_>$.
In other words, the stoichiometric coefficient of a species on the left-hand side of the reaction
equals the exponent of the corresponding concentration in the rate function.
It remains to formally introduce the rate constants.
\begin{definition} \label{MAK}
A mass action system $(\mathscr{S},\mathscr{C},\mathscr{R},k)$
is a chemical reaction network $(\mathscr{S},\mathscr{C},\mathscr{R})$
together with a vector $k \in \RR^\mathscr{R}_>$ of rate constants.
\end{definition}

\begin{definition} \label{ODE}
The ordinary differential equation (ODE)
associated with a mass action system $(\mathscr{S},\mathscr{C},\mathscr{R},k)$
is defined as
\begin{equation*}
\dd{c}{t}
= r(c)
\end{equation*}
with the species formation rate
\begin{equation*}
r(c)
= \sum_{y \to y' \in \mathscr{R}} k_{y \to y'} \, c^y \, (y' - y) \, .
\end{equation*}
\end{definition}

In order to rewrite the species formation rate,
we use the unit vectors $\omega_y \in \RR^\mathscr{C}$ corresponding to complexes $y \in \mathscr{C}$
and define

\vspace{1ex}
\begin{itemize}
\item
a linear map%
\footnote{
The corresponding matrix amounts to $Y_{sy}=y_s$.
}
$Y \colon \RR^\mathscr{C} \to \RR^\mathscr{S}$ with $Y \omega_y = y$,
\item
a nonlinear map $\Psi \colon \RR^\mathscr{S}_\ge \to \RR^\mathscr{C}$,
$c \mapsto \displaystyle{ \sum_{y \in \mathscr{C}}} c^{y} \, \omega_y$, and
\item
a linear map%
\footnote{
The corresponding matrix amounts to
$A_{yy'} = K_{y'y} - \delta_{yy'} \sum_{y'' \in \mathscr{C}} K_{yy''}$,
where $K \in \RR^{\mathscr{C} \times \mathscr{C}}$
with $K_{yy'} = k_{y \to y'}$ if $y \to y' \in \mathscr{R}$ and $K_{yy'} = 0$ otherwise.
}
$A \colon \RR^\mathscr{C} \to \RR^\mathscr{C}$,
$x \mapsto \displaystyle{\sum_{y \to y' \in \mathscr{R}}} k_{y \to y'} \, x_y \, (\omega_{y'}-\omega_y)$.
\end{itemize}
Now, the species formation rate can be decomposed as
\begin{align} \label{sfr}
r(c) &= \sum_{y \to y' \in \mathscr{R}} k_{y \to y'} \, c^y \, (y' - y) \\
& =   Y \sum_{y \to y' \in \mathscr{R}} k_{y \to y'} \, c^y \, (\omega_{y'} - \omega_y) \nonumber \\
& =   Y \sum_{y \to y' \in \mathscr{R}} k_{y \to y'} \, \Psi(c)_y \, (\omega_{y'} - \omega_y) \nonumber \\
& = Y A_{\,} \Psi(c) \, . \nonumber
\end{align}

Equilibria of the ODE associated with a mass action system
satisfying $A_{\,} \Psi(c) = 0$ and $c>0$ are called {\em complex balancing equilibria}.
The possibility of other (positive) equilibria
suggests the definition of the {\em deficiency} of a mass action system.
\begin{definition} \label{EQU}
Let $(\mathscr{S},\mathscr{C},\mathscr{R},k)$ be a mass action system.
The set of complex balancing equilibria
is defined as
\begin{equation*}
Z = \{ c \in \RR^\mathscr{S}_> \st A_{\,} \Psi(c) = 0 \} \, .
\end{equation*}
The deficiency of the system is defined as
\begin{equation*}
\delta = \dim(\ker(Y) \cap \im(A)) \, .
\end{equation*}
\end{definition}

Originally, the deficiency was defined differently.
As we will see in Proposition \ref{prodef},
the two definitions coincide under certain conditions on the network structure.
In Fig.\ \ref{fig:def}, we summarize the definitions associated with a mass action system
and depict their dependencies.
\begin{figure}[htb]
\begin{displaymath}
\boxed{
\xymatrix{
\mathscr{S} & \mathscr{C} \ar[ddl] \ar[d] \ar[dr] & \\
& Y \ar[d] & \Psi \ar[d] \\
S & \delta & Z \\
& A \ar[u] \ar[ur] & \\
& \mathscr{R} \ar[uul] \ar[u] & k \ar[ul]
}
}
\end{displaymath}
\begin{equation*}
\dd{c}{t} = Y A_{\,} \Psi(c)
\end{equation*}
\caption{
The mass action system
$(\mathscr{S},\mathscr{C},\mathscr{R},k)$:
Associated definitions and their dependencies. (Definitions at arrowheads depend on tails.)
\label{fig:def}
}
\end{figure}


\subsection*{Results}

Now we are in a position to present several results of CRNT
related to the Deficiency Zero Theorem.
(The results are due to Horn, Jackson, and Feinberg \cite{HornJackson1972,Horn1972,Feinberg1972}.
For proofs, we refer the reader to the surveys \cite{Feinberg1979,Feinberg1995a} or \cite{Gunawardena}.)
As we will see later,
corresponding statements also hold in the case of generalized mass action kinetics.
We start with a foundational linear algebra result,
which can be proved using the Perron-Frobenius Theorem.

\newpage

\begin{theorem} \label{kerA}
Let $(\mathscr{S},\mathscr{C},\mathscr{R},k)$ be a mass action system
with the associated map $A$,
and let $T_1,\ldots,T_t \subseteq \mathscr{C}$ be the terminal strong linkage classes.
Then:
\begin{remunerate}
\item
for $i=1,\ldots,t \colon \exists \chi_i \in \RR^\mathscr{C}_\ge$ with $\supp(\chi_i)=T_i$
\item
$\ker(A)=\spann \{ \chi_1,\ldots,\chi_t \}$
\item
$\dim(\ker(A))=t$
\end{remunerate}
\end{theorem}

\vspace{1ex}
The next result is an immediate consequence of Theorem \ref{kerA}.
\begin{corollary} \label{cor}
Let $(\mathscr{S},\mathscr{C},\mathscr{R})$ be a chemical reaction network.
If there exist rate constants $k$
such that the mass action system $(\mathscr{S},\mathscr{C},\mathscr{R},k)$ has a complex balancing equilibrium,
then $(\mathscr{S},\mathscr{C},\mathscr{R})$ is weakly reversible.
\end{corollary}

\vspace{1ex}
If each linkage class contains exactly one terminal strong linkage class,
the deficiency is independent of the rate constants
and can be computed from basic parameters of the chemical reaction network.
The resulting formula was the original definition of the deficiency.
\begin{proposition} \label{prodef}
If a chemical reaction network $(\mathscr{S},\mathscr{C},\mathscr{R})$ is weakly reversible
(or more generally if $t=l$), then,
for all rate constants $k$,
the deficiency of the mass action system $(\mathscr{S},\mathscr{C},\mathscr{R},k)$
is given by $\delta = m-l-s$,
where $m$ is the number of complexes,
$l$ is the number of linkage classes,
and $s$ is the dimension of the stoichiometric subspace.
\end{proposition}

\vspace{1ex}
In the case of deficiency zero,
weak reversibility guarantees the existence of complex balancing equilibria.
\begin{proposition} \label{prodefzero}
If a chemical reaction network $(\mathscr{S},\mathscr{C},\mathscr{R})$ is weakly reversible and $\delta = 0$, then,
for all rate constants $k$,
the mass action system $(\mathscr{S},\mathscr{C},\mathscr{R},k)$
has a complex balancing equilibrium.
\end{proposition}

\vspace{1ex}
Theorem \ref{kerA} further implies
that the set of complex balancing equilibria
can be parametrized by the orthogonal of the stoichiometric subspace.
\begin{proposition} \label{proZ}
Let $(\mathscr{S},\mathscr{C},\mathscr{R},k)$ be a mass action system
with nonempty set $Z$ of complex balancing equilibria. 
Then
\begin{equation*}
Z = \{ c \in \RR^\mathscr{S}_> \st \ln(c)-\ln(c^*) \in S^\bot \}
= \{ c^* \mal \e^v \st v \in S^\bot \}
\end{equation*}
for any $c^* \in Z$.
\end{proposition}

\vspace{1ex}
Finally, we recall a result
concerned with the existence and uniqueness of a complex balancing equilibrium
in every stoichiometric compatibility class.
It can be proved using methods from convex analysis.
\begin{theorem} \label{thm}
Let $(\mathscr{S},\mathscr{C},\mathscr{R},k)$ be a mass action system
with nonempty set $Z$ of complex balancing equilibria.
Then $Z$ meets every stoichiometric compatibility class in exactly one point.
\end{theorem}

\vspace{1ex}
In Section \ref{sec:cbe},
we study the conditions under which a result analogous to Theorem \ref{thm}
holds in the case of generalized mass action kinetics.


\subsection{Generalized mass action systems} \label{sec:gmas}

Chemical reactions occur between entire molecules
such that the stoichiometric coefficients are integers.
Under the assumption of mass action kinetics,
the rate functions are monomials in the concentrations of the reactant species.
However,
in Definition \ref{CRN}
we allowed nonnegative real stoichiometric coefficients
and hence ``generalized monomials'' as rate functions,
since all results presented above also hold in this generality.
This observation can be used to account for generalized mass action kinetics.
We outline two different approaches the second of which is the focus of this paper.

In the first approach \cite{HornJackson1972},
chemical reactions are redefined as pseudo-reactions
with the same net balance, but real stoichiometric coefficients.
For example, the reaction
\begin{equation*}
n_A A + n_B B \to n_C C
\end{equation*}
with $n_A$, $n_B$, $n_C \in \NN$ can be redefined as
\begin{equation*}
\nu_A A + \nu_B B + \nu_C C \to (\nu_A-n_A) A + (\nu_B-n_B) B + (\nu_C + n_C) C
\end{equation*}
with $\nu_A, \nu_B, \nu_C \in \RR_\ge$
and rate function $k \, [A]^{\nu_A} [B]^{\nu_B} [C]^{\nu_C}$.
The redefinition of chemical reactions does not affect the stoichiometric subspace,
however, it entails a new (and typically larger) set of complexes
and hence a new mass action system (with different properties).
For example, consider the (weakly) reversible chemical reaction network
\begin{equation*}
A + B \rightleftharpoons C
\end{equation*}
with two complexes and one linkage class.
Since the stoichiometric subspace $S=\spann \{ (-1,-1,1)^T \}$ has dimension one,
we obtain $\delta = 2 - 1 - 1 = 0$ by Proposition~\ref{prodef}.
In order to account for generalized mass action kinetics
specified by the rate functions $k_{A + B \to C} [A]^a [B]^b$ and $k_{C \to A + B} [C]^c$
with $a,b,c \in\RR_>$,
the system can be redefined by the pseudo-reactions
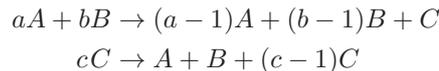
\begin{align*}
a A + b B &\to (a - 1) A + (b - 1) B + C \\
c \hspace{1pt} C &\to A + B + (c - 1) C
\end{align*}
with four complexes and two linkage classes.
This new system is not weakly reversible and has deficiency $\delta = 4 - 2 - 1 = 1$, again by Proposition~\ref{prodef}.

In this paper,
we present a different way to account for generalized mass action kinetics.
Most importantly,
we disentangle the definition of the rate functions
from the stoichiometric coefficients.
In particular,
we keep the integer stoichiometric coefficients,
but we allow ``generalized monomials'' as rate functions,
in which the exponents of the concentrations can be arbitrary nonnegative real numbers.
More formally,
we do not change the chemical reaction network,
but we associate with each complex a so-called {\em \kin complex},
which determines the exponents of the concentrations in the rate function of the respective reaction.
In the above example,
we associate the \kin complexes $a A + b B$ and $c \hspace{1pt} C$ with $A + B$ and $C$,
thereby specifying the rate functions $k_{A + B \to C} [A]^a [B]^b$ and $k_{C \to A + B} [C]^c$.
We obtain the following network, where we indicate association of kinetic complexes by dots:
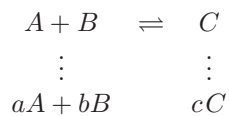
\begin{equation*}
\begin{array}{ccc}
A + B & \rightleftharpoons & C \\
\vdots && \vdots \\
a A + b B && c \hspace{1pt} C
\end{array}
\end{equation*}
For an arbitrary chemical reaction network with generalized mass action kinetics,
the rate function $\mathscr{K}_{y \to y'} \colon \RR^\mathscr{S}_\ge \to \RR_\ge$
corresponding to reaction $y \to y' \in \mathscr{R}$
is given by $\mathscr{K}_{y \to y'}(c) = k_{y \to y'} \, c^{\tilde{y}}$,
where $\tilde{y}$ is the \kin complex associated with $y$.
\begin{definition} \label{genCRN}
A generalized chemical reaction network $(\mathscr{S},\mathscr{C},\tilde{\mathscr{C}},\mathscr{R})$
is a chemical reaction network $(\mathscr{S},\mathscr{C},\mathscr{R})$
together with a family $\tilde{\mathscr{C}}=(x_y)_{y\in \mathscr{C}}$ in $\RR_\ge^\mathscr{S}$ of \kin complexes,
where $\lvert \{x_y \st y\in \mathscr{C} \} \rvert = \lvert  \mathscr{C} \rvert$.
We write $\tilde{y}=x_y$ for the \kin complex associated with the complex $y\in \mathscr{C}$.
\end{definition}

A generalized chemical reaction network $(\mathscr{S},\mathscr{C},\tilde{\mathscr{C}},\mathscr{R})$
contains the chemical reaction network $(\mathscr{S},\mathscr{C},\mathscr{R})$;
moreover, it entails the fictitious chemical reaction network
$(\mathscr{S},\underline{\tilde{\mathscr{C}}},\underline{\mathscr{R}})$
where the set $\underline{\tilde{\mathscr{C}}} = \{\tilde{y} \st y \in \mathscr{C} \}$
has the same cardinality as $\mathscr{C}$ (by definition)
and the relation $\underline{\mathscr{R}}$ is isomorphic to $\mathscr{R}$,
i.e.\ $\tilde{y} \to \tilde{y}'\in \underline{\mathscr{R}}$ whenever $y \to y' \in \mathscr{R}$.
Hence the networks $(\mathscr{S},\mathscr{C},\mathscr{R})$
and $(\mathscr{S},\underline{\tilde{\mathscr{C}}},\underline{\mathscr{R}})$
give rise to the same directed graph (up to renaming of vertices).
A generalized chemical reaction network $(\mathscr{S},\mathscr{C},\tilde{\mathscr{C}},\mathscr{R})$ is called weakly reversible
if $(\mathscr{S},\mathscr{C},\mathscr{R})$ is weakly reversible.
Also the definitions of the stoichiometric subspace and the stoichiometric compatibility classes
carry over from $(\mathscr{S},\mathscr{C},\mathscr{R})$
to $(\mathscr{S},\mathscr{C},\tilde{\mathscr{C}},\mathscr{R})$, cf.\ Definition \ref{SSS}.
Additionally, we introduce the \kinsub subspace of a generalized chemical reaction network,
which coincides with the stoichiometric subspace of the fictitious network.

\begin{definition}
Let $(\mathscr{S},\mathscr{C},\tilde{\mathscr{C}},\mathscr{R})$ be a generalized chemical reaction network.
The \kinsub subspace is defined as
\begin{equation*}
\tilde{S} = \spann \{ \tilde{y'}-\tilde{y} \st y \to y' \in \mathscr{R} \} \, .
\end{equation*}
\end{definition}

For consistency, the name {\em \kin subspace} would be more appropriate for $\tilde{S}$ but
this name has already been given to a certain subspace of the stoichiometric subspace \cite{FeinbergHorn1977},
which coincides with the stoichiometric subspace if $t=l$.

For later use, we introduce the maps

\vspace{1ex}
\begin{itemize}
\item
$\tilde{Y} \colon \RR^\mathscr{C} \to \RR^\mathscr{S}$ with $\tilde{Y} \omega_y = \tilde{y}$ and
\item
$\tilde{\Psi} \colon \RR^\mathscr{S}_\ge \to \RR^\mathscr{C}$,
$\displaystyle{ c \mapsto \sum_{y \in \mathscr{C}} c^{\tilde{y}} \, \omega_y }$ ,
\end{itemize}
where we identify $\RR^\mathscr{C}$ and $\RR^{\underline{\tilde{\mathscr{C}}}}$.

\begin{definition} \label{genMAK}
A generalized mass action system $(\mathscr{S},\mathscr{C},\tilde{\mathscr{C}},\mathscr{R},k)$ is a
generalized chemical reaction network $(\mathscr{S},\mathscr{C},\tilde{\mathscr{C}},\mathscr{R})$
together with a vector $k \in \RR^\mathscr{R}_>$ of rate constants.
\end{definition}

\begin{definition} \label{genODE}
The ordinary differential equation (ODE)
associated with a generalized mass action system $(\mathscr{S},\mathscr{C},\tilde{\mathscr{C}},\mathscr{R},k)$
is defined as
\begin{equation*}
\dd{c}{t}
= \tilde{r}(c)
\end{equation*}
with the species formation rate
\begin{equation*}
\tilde{r}(c)
= \sum_{y \to y' \in \mathscr{R}} k_{y \to y'} \, c^{\tilde{y}} \, (y' - y) \, .
\end{equation*}
\end{definition}

As in Eqn.\ \eqref{sfr},
we can decompose the species formation rate of a generalized mass action system as
\begin{align*}
\tilde{r}(c) &= Y A_{\,} \tilde{\Psi}(c) \, .
\end{align*}

In analogy to Definition \ref{EQU}, equilibria satisfying $A_{\,} \tilde{\Psi}(c) = $ and $c > 0$
are called complex balancing equilibria;
they coincide with the complex balancing equilibria of the fictitious mass action system
$(\mathscr{S},\underline{\tilde{\mathscr{C}}},\underline{\mathscr{R}},k)$.
The deficiency, which quantifies the possibility of other equilibria,
coincides with the deficiency of the mass action system $(\mathscr{S},\mathscr{C},\mathscr{R},k)$.

\begin{definition}  \label{genEQU}
Let $(\mathscr{S},\mathscr{C},\tilde{\mathscr{C}},\mathscr{R},k)$ be a generalized mass action system.
The set of complex balancing equilibria is defined as
\begin{equation*}
\tilde{Z} = \{ c \in \RR^\mathscr{S}_> \st A_{\,} \tilde{\Psi}(c) = 0 \}
\end{equation*}
and the deficiency as
\begin{equation*}
\delta = \dim(\ker(Y) \cap \im(A)) \, .
\end{equation*}
\end{definition}

It remains to introduce the \kin deficiency,
which coincides with the deficiency of the fictitious system.

\begin{definition}
Let $(\mathscr{S},\mathscr{C},\tilde{\mathscr{C}},\mathscr{R},k)$ be a generalized mass action system.
The \kin deficiency is defined as
\begin{equation*}
\tilde{\delta} = \dim(\ker(\tilde{Y}) \cap \im(A)) \, .
\end{equation*}
\end{definition}

In Fig.\ \ref{fig:compare},
we summarize the definitions associated with a generalized mass action system
and depict their dependencies.
From the mass action system $(\mathscr{S},\mathscr{C},\mathscr{R},k)$,
we keep the stoichiometric subspace $S$ and the deficiency $\delta$,
whereas we use all definitions associated with the fictitious mass action system
$(\mathscr{S},\underline{\tilde{\mathscr{C}}},\underline{\mathscr{R}},k)$;
in particular, the \kinsub subspace $\tilde{S}$, the \kin deficiency $\tilde{\delta}$,
and the set $\tilde{Z}$ of complex balancing equilibria.

\begin{figure}[htb]
\begin{displaymath}
\boxed{
\xymatrix{
\mathscr{S} & \mathscr{C} \ar[ddl] \ar[d] &   & & \underline{\tilde{\mathscr{C}}} \ar[ddl] \ar[d] \ar[dr] & \\
& Y \ar[d] &                                  & & \tilde{Y} \ar[d] & \tilde{\Psi} \ar[d] \\
S & \delta &                                  & \tilde{S} & \tilde{\delta} & \tilde{Z} \\
& A \ar[u] &                                  & & A \ar[u] \ar[ur] & \\
& \mathscr{R} \ar[uul] \ar[u] & k \ar[ul]     & & \underline{\mathscr{R}} \ar[uul] \ar[u] & k \ar[ul]
}
}
\end{displaymath}
\begin{equation*}
\dd{c}{t} = Y A_{\,} \tilde{\Psi}(c)
\end{equation*}
\caption{
The generalized mass action system
$(\mathscr{S},\mathscr{C},\tilde{\mathscr{C}},\mathscr{R},k)$:
Associated definitions and their dependencies.
(For better readability, $k$ and $A$ are plotted twice.)
\label{fig:compare}
}
\end{figure}


\subsection*{Results}

Now we return to the results of CRNT
that have been derived for mass action systems.
Since Theorem \ref{kerA} is concerned with the kernel of the linear map $A$,
the underlying kinetics is not relevant at all.
But also Corollary \ref{cor} and Propositions \ref{prodef}--\ref{proZ}
carry over easily to generalized mass action systems
if we consider the fictitious chemical reaction network $(\mathscr{S},\underline{\tilde{\mathscr{C}}},\underline{\mathscr{R}})$
and the fictitious mass action system $(\mathscr{S},\underline{\tilde{\mathscr{C}}},\underline{\mathscr{R}},k)$ defined above.
For reference, we present the analogous results.
\begin{proposition}
Let $(\mathscr{S},\mathscr{C},\mathscr{R})$ be a chemical reaction network.
If there exists a generalized mass action system $(\mathscr{S},\mathscr{C},\tilde{\mathscr{C}},\mathscr{R},k)$ 
with a complex balancing equilibrium,
then $(\mathscr{S},\mathscr{C},\mathscr{R})$ is weakly reversible.
\end{proposition}

\begin{proof}
Assume that $(\mathscr{S},\mathscr{C},\tilde{\mathscr{C}},\mathscr{R},k)$
and hence the mass action system $(\mathscr{S},\underline{\tilde{\mathscr{C}}},\underline{\mathscr{R}},k)$
have a complex balancing equilibrium.
By Corollary \ref{cor}, the chemical reaction network $(\mathscr{S},\underline{\tilde{\mathscr{C}}},\underline{\mathscr{R}})$
and hence $(\mathscr{S},\mathscr{C},\mathscr{R})$ are weakly reversible.
\end{proof}

\vspace{1ex}
\begin{proposition} \label{prodef'}
If a chemical reaction network $(\mathscr{S},\mathscr{C},\mathscr{R})$ is weakly reversible
(or more generally if $t=l$), then
the deficiencies of any generalized mass action system $(\mathscr{S},\mathscr{C},\tilde{\mathscr{C}},\mathscr{R},k)$
are given by $\delta = m-l-s$ and $\tilde{\delta} = m-l-\tilde{s}$,
where $m$ is the number of complexes,
$l$ is the number of linkage classes,
$s$ is the dimension of the stoichiometric subspace,
and $\tilde{s}$ is the dimension of the \kinsub subspace.
\end{proposition}

\begin{proof}
Assume that $(\mathscr{S},\mathscr{C},\mathscr{R})$ 
and hence the chemical reaction network $(\mathscr{S},\underline{\tilde{\mathscr{C}}},\underline{\mathscr{R}})$
arising from $(\mathscr{S},\mathscr{C},\tilde{\mathscr{C}},\mathscr{R},k)$
are weakly reversible (or more generally that $t=l$).
The deficiency of the generalized mass action system
equals the deficiency of $(\mathscr{S},\mathscr{C},\mathscr{R},k)$,
and the \kin deficiency equals the deficiency of $(\mathscr{S},\underline{\tilde{\mathscr{C}}},\underline{\mathscr{R}},k)$.
By Proposition \ref{prodef}, the deficiencies of the two mass action systems are given by the  formulas stated.
\end{proof}



\begin{proposition} \label{prodefzero'}
If a generalized chemical reaction network $(\mathscr{S},\mathscr{C},\tilde{\mathscr{C}},\mathscr{R})$
is weakly reversible and $\tilde{\delta} = 0$, then
any generalized mass action system $(\mathscr{S},\mathscr{C},\tilde{\mathscr{C}},\mathscr{R},k)$
has a complex balancing equilibrium.
\end{proposition}

\begin{proof}
Assume that $(\mathscr{S},\mathscr{C},\mathscr{R})$ 
and hence the chemical reaction network $(\mathscr{S},\underline{\tilde{\mathscr{C}}},\underline{\mathscr{R}})$
arising from $(\mathscr{S},\mathscr{C},\tilde{\mathscr{C}},\mathscr{R},k)$
are weakly reversible. Additionally, assume $\tilde{\delta} = 0$.
By Proposition \ref{prodefzero},
the mass action system $(\mathscr{S},\underline{\tilde{\mathscr{C}}},\underline{\mathscr{R}},k)$
and hence $(\mathscr{S},\mathscr{C},\tilde{\mathscr{C}},\mathscr{R},k)$
have a complex balancing equilibrium.
\end{proof}

\vspace{1ex}
\begin{proposition} \label{proZ'}
Let $(\mathscr{S},\mathscr{C},\tilde{\mathscr{C}},\mathscr{R},k)$ be a generalized mass action system with
nonempty set $\tilde{Z}$ of complex balancing equilibria. 
Then
\begin{equation*}
\tilde{Z} = \{ c \in \RR^\mathscr{S}_> \st \ln(c)-\ln(c^*) \in \tilde{S}^\bot \}
= \{ c^* \mal \e^{\tilde{v}} \st \tilde{v} \in \tilde{S}^\bot \} 
\end{equation*}
for any $c^* \in \tilde{Z}$.
\end{proposition}

\begin{proof}
The complex balancing equilibria
of the mass action system  $(\mathscr{S},\underline{\tilde{\mathscr{C}}},\underline{\mathscr{R}},k)$
coincides with $\tilde{Z}$,
and its stoichiometric subspace
coincides with $\tilde{S}$, the \kinsub subspace of $(\mathscr{S},\mathscr{C},\tilde{\mathscr{C}},\mathscr{R},k)$.
By Proposition \ref{proZ},
the nonempty set $\tilde{Z}$ is given by the formula stated.
\end{proof}

One might conjecture that also Theorem \ref{thm} holds for generalized mass action systems.
However, an analogous result depends on both
the complexes $\mathscr{C}$ and the \kin complexes $\tilde{\mathscr{C}}$,
where $\mathscr{C}$ determines the stoichiometric subspace $S$
(and hence the stoichiometric compatibility classes $(c'+S)_\geq$),
whereas $\tilde{\mathscr{C}}$ determines the set $\tilde{Z}$ of complex balancing equilibria
(and the related \kinsub subspace $\tilde{S}$).
It turns out that the result depends on additional assumptions
concerning the sign vectors of the subspaces $S$ and $\tilde{S}$,
see Theorem \ref{genthm}.


\section{Complex balancing equilibria} \label{sec:cbe}

In the following,
we consider a generalized mass action system
$(\mathscr{S},\mathscr{C},\tilde{\mathscr{C}},\mathscr{R},k)$
with stoichiometric subspace $S$,
\kinsub subspace $\tilde{S}$,
and nonempty set $\tilde{Z}$ of complex balancing equilibria.

From Proposition \ref{proZ'}
we know that $\tilde{Z} = \{ c^* \mal \e^{\tilde{v}} \st \tilde{v} \in \tilde{S}^\bot \}$ for any $c^* \in \tilde{Z}$.
We provide necessary and sufficient conditions
such that in every stoichiometric compatibility class $(c'+S)_\geq$
there is at most one complex balancing equilibrium.
Moreover,
we provide sufficient conditions
such that in every stoichiometric compatibility class
there is at least one complex balancing equilibrium.

The question of uniqueness is answered by the following result for arbitrary subspaces $S$ and $\tilde{S}$.
It involves the corresponding sets of sign vectors denoted by $\sigma(S)$ and  $\sigma(\tilde{S})$;
for the definition of sign vectors and related notions we refer the reader to the Appendix.
We note that sign vectors also appear in the study of multiple equilibria
that are not necessarily complex balancing 
\cite{Feinberg1995b,PerezMillanDickensteinShiuConradi2012}.
\begin{proposition} \label{unique}
Let $S,\tilde{S}$ be subspaces of $\RR^n$.
Then the two statements are equivalent:
\begin{remunerate}
\item
For all $c^*>0$ and $c'>0$, the intersection
$(c'+S)_\geq \cap \{ c^* \mal \e^{\tilde{v}} \st \tilde{v} \in \tilde{S}^\bot \}$
contains at most one element.
\item
$\sigma(S) \cap \sigma(\tilde{S}^\bot) = \{0\}$.
\end{remunerate}
\end{proposition}

\vspace{1ex}
\begin{proof}
$(\neg 1 \Rightarrow \neg 2)$:
Suppose there exist $u^1 \not= u^2 \in S$ and $\tilde{v}_1 \not= \tilde{v}_2 \in \tilde{S}^\bot$
such that $c' + u^1 = c^* \mal \e^{\tilde{v}_1}$ and $c' + u^2 = c^* \mal \e^{\tilde{v}_2}$
(for a certain $c'$ and a certain $c^*$).
Then $u^1 - u^2 = c^* \mal ( \e^{\tilde{v}_1} - \e^{\tilde{v}_2} )$
and by the monotonicity of the exponential function
\[
\sigma ( \underbrace{ u^1 - u^2 }_{\in \, S} )
= \sigma ( c^* \mal ( \e^{\tilde{v}^1} - \e^{\tilde{v}^2} ) )
= \sigma ( \e^{\tilde{v}^1} - \e^{\tilde{v}^2} )
= \sigma ( \underbrace{\tilde{v}^1 - \tilde{v}^2}_{\in \, \tilde{S}^\bot} ) \, .
\]
Hence
$\sigma(S) \cap \sigma(\tilde{S}^\bot) \neq \{0\}$.

$(\neg 2 \Rightarrow \neg 1)$:
Suppose that $0 \not= \tau \in \sigma(S) \cap \sigma(\tilde{S}^\bot)$.
Then there exist $u \in S$ and $\tilde{v}^1 \in \tilde{S}^\bot$
such that $\sigma(u) = \sigma(\tilde{v}^1) = \tau$.
Further, let $\tilde{v}^2 = \frac{1}{2} \tilde{v}^1$.
Then $\sigma(\tilde{v}^1 - \tilde{v}^2) = \tau$
and
\[
\sigma(u)
= \sigma(\tilde{v}^1 - \tilde{v}^2)
= \sigma ( \e^{\tilde{v}^1} - \e^{\tilde{v}^2} )
= \sigma ( c^* \mal ( \e^{\tilde{v}^1} - \e^{\tilde{v}^2} ) )
\]
for all $c^*>0$.
In particular, there is $c^*$
such that $u = c^* \mal ( \e^{\tilde{v}^1} - \e^{\tilde{v}^2} )$.
With $c'= c^* \mal \e^{\tilde{v}^1}$,
one has $c' - u = c^* \mal \e^{\tilde{v}^2}$
and hence both $c'$ and $c'-u$ are elements of $(c'+S)_\geq \cap \{ c^* \mal \e^{\tilde{v}} \st \tilde{v} \in \tilde{S}^\bot \}$.
\end{proof}

It follows in particular that if the sign vectors are equal, $\sigma(S) = \sigma(\tilde{S})$,
complex balancing equilibria are unique (in a stoichiometric compatibility class)
since then
\[
\sigma(S) \cap \sigma(\tilde{S}^\bot) = \sigma(S) \cap \sigma(\tilde{S})^\bot = \sigma(S) \cap \sigma(S)^\bot = \{0\}
\]
using Eqn.~\eqref{eq:orthsubspace}.
Note that this is only a sufficient condition;
for example, with $S=\spann \{ (-1,1) \}$ and $\tilde{S} = \spann \{ (-1,0) \}$,
we have $\sigma(S) \cap \sigma(\tilde{S}^\bot) = \{ 0 \}$ but $\sigma(S) \not = \sigma(\tilde{S})$.
However, it includes classical mass action kinetics where $S = \tilde{S}$
and each stoichiometric compatibility class contains at most one complex balancing equilibrium.
On the other hand, if $\sigma(S) \cap \sigma(\tilde{S}^\bot) \neq \{0\}$ and the underlying network is weakly reversible,
then such a generalized chemical reaction network has the capacity for multiple complex balancing equilibria,
as shown in the following result.

\begin{proposition} \label{notunique}
If a generalized chemical reaction network $(\mathscr{S},\mathscr{C},\tilde{\mathscr{C}},\mathscr{R})$
is weakly reversible and $\sigma(S) \cap \sigma(\tilde{S}^\bot) \neq \{0\}$,
there exist rate constants $k$
such that the generalized mass action system $(\mathscr{S},\mathscr{C},\tilde{\mathscr{C}},\mathscr{R},k)$
has more than one complex balancing equilibrium in some stoichiometric compatibility class.
\end{proposition}

\begin{proof}
Let $\sigma(S) \cap \sigma(\tilde{S}^\bot) \neq \{0\}$.
By Proposition \ref{unique}, there exist $c^*>0$ and $c'> 0$
such that $(c'+S)_\geq \cap \{ c^* \mal \e^{\tilde{v}} \st \tilde{v} \in \tilde{S}^\bot \}$
contains more than one element.
Using Proposition \ref{proZ'},
it remains to show that there exist rate constants $k\in \RR^\mathscr{R}_>$ such that $c^*$
is a complex balancing equilibrium of $(\mathscr{S},\mathscr{C},\tilde{\mathscr{C}},\mathscr{R},k)$,
i.e.\
\begin{equation*}
A \, \tilde{\Psi}(c^*) =
\sum_{y \to y' \in \mathscr{R}} k_{y \to y'} \, (c^*)^{\tilde{y}} \, (\omega_{y'} - \omega_{y}) =
0 \, .
\end{equation*}
Since $(\mathscr{S},\mathscr{C},\tilde{\mathscr{C}},\mathscr{R})$
and hence $(\mathscr{S},\mathscr{C},\mathscr{R})$ are weakly reversible,
this is guaranteed by Lemma \ref{circulation}.
\end{proof}

In the proof of Proposition \ref{notunique},
we use the following result.
\begin{lemma} \label{circulation}
Let $(\mathscr{S},\mathscr{C},\mathscr{R})$ be a chemical reaction network.
Then, the following statements are equivalent:
\begin{remunerate}
\item $(\mathscr{S},\mathscr{C},\mathscr{R})$ is weakly reversible.
\item There exists $k \in \RR^\mathscr{R}_>$
such that $\sum_{y \to y' \in \mathscr{R}} k_{y \to y'} \, (\omega_{y'} - \omega_{y}) = 0$,
where $\omega_{y} \in \RR^\mathscr{C}$ denotes the unit vector corresponding to $y \in \mathscr{C}$.
\end{remunerate}
\end{lemma}

\begin{proof}
$(1 \Rightarrow 2)$: By weak reversibility,
there exists a cycle $y \to y' \to \ldots \to y$ for each reaction $y \to y' \in \mathscr{R}$ 
and we denote the set of reactions involved in this cycle by $C_{y \to y'}$.
Clearly, $\sum_{z \to z' \in {C_{y \to y'}}} (\omega_{z'} - \omega_{z}) = 0$
and hence
\begin{equation*}
\sum_{y \to y' \in \mathscr{R}} \, \sum_{z \to z' \in {C_{y \to y'}}} (\omega_{z'} - \omega_{z})
= \sum_{y \to y' \in \mathscr{R}} k_{y \to y'} \, (\omega_{y'} - \omega_{y})
= 0 \, ,
\end{equation*}
where $k_{y \to y'}>0$ records in how many cycles the reaction $y \to y'$ appears.

$(2 \Rightarrow 1)$:
We write $\sum_{y \to y' \in \mathscr{R}} k_{y \to y'} \, (\omega_{y'} - \omega_{y}) = A \, \Omega$
with $\Omega = (1, 1, \ldots, 1)^T \in \RR^\mathscr{C}_>$.
By Theorem \ref{kerA},
if $A \, \Omega = 0$, then $(\mathscr{S},\mathscr{C},\mathscr{R})$ is weakly reversible.
\end{proof}

The second implication is a basic fact from CRNT \cite{Horn1972,Feinberg1995a}.

\subsection{The map $F$} \label{sec:F}

In order to study uniqueness and existence
in a common framework,
we rephrase the problem.
We suppose that $\mathscr{S}$ contains $n$ species and fix an order among them.
Then we can identify $\RR^\mathscr{S}$ with $\RR^n$
such that $S$, $\tilde{S} \subseteq \RR^n$.
Further,
let $V=(v^1, \ldots , v^d)$ and $\tilde{V}=(\tilde{v}^1, \ldots , \tilde{v}^{\tilde{d}})$
be bases for $S^\bot$ and $\tilde{S}^\bot$, respectively.
In other words, $S^\bot = \im(V)$ and $\dim(S^\bot) = d$
and analogously $\tilde{S}^\bot = \im(\tilde{V})$ and $\dim(\tilde{S}^\bot) = \tilde{d}$.

An element in $(c'+S)_\geq \cap \{ c^* \mal \e^{\tilde{v}} \st \tilde{v} \in \tilde{S}^\bot \}$
corresponds to $u \in S$ and $\tilde{v} \in \tilde{S}^\bot$
such that $c^* \mal \e^{\tilde{v}} = c' + u$
or equivalently to $\lambda \in \RR^{\tilde{d}}$
such that
\begin{equation*}
\langle c^* \mal \expp{\sum_{j=1}^{\tilde{d}} \lambda_j \tilde{v}^j} , v^i \rangle = \langle c' , v^i \rangle
\quad \text{for} \quad
i=1, \ldots, d \, .
\end{equation*}
Hence, provided $c^* \in \tilde{Z}$,
uniqueness and existence (of a complex balancing equilibrium in every stoichiometric compatibility class)
correspond to injectivity and surjectivity of the following map:
\begin{align} \label{Fequ}
F \colon & \RR^{\tilde{d}} \to C^\circ \subseteq \RR^d \\
& \lambda \mapsto F(\lambda)
\quad \text{with} \quad
(F(\lambda))_i = \langle c^* \mal \expp{\sum_{j=1}^{\tilde{d}} \lambda_j \tilde{v}^j} , v^i \rangle \, , \nonumber
\end{align}
where $c^* >0$ and
\begin{equation*}
C^\circ = \{ \gamma \in \RR^d \st \gamma_i = \langle c' , v^i \rangle , \, c' \in \RR^n_> \} \, .
\end{equation*}
Note that $F$ depends on $c^*$.
It is instructive to reformulate the definition of $F$.
To this end, we express the columns of $V$ and $\tilde{V}$ by its rows,
\begin{align*}
V &=
( v^1 , \ldots , v^d ) =
( w^1 , \ldots , w^n )^T \\
\tilde{V} &=
( \tilde{v}^1 , \ldots , \tilde{v}^{\tilde{d}} ) =
( \tilde{w}^1 , \ldots , \tilde{w}^n )^T \, ,
\end{align*}
or equivalently $v^j_i = w^i_j$ and $\tilde{v}^j_i = \tilde{w}^i_j$,
and obtain:
\begin{align*}
(F(\lambda))_i
&= \langle c^* \mal \expp{\sum_{j=1}^{\tilde{d}} \lambda_j \tilde{v}^j} , v^i \rangle
= \sum_{k=1}^n c^*_k \, \expp{\sum_{j=1}^{\tilde{d}} \lambda_j \tilde{v}^j_k} \, v^i_k \\
&= \sum_{k=1}^n c^*_k \, \expp{\sum_{j=1}^{\tilde{d}} \lambda_j \tilde{w}^k_j} \, w^k_i
= \sum_{k=1}^n c^*_k \, \e^{\langle \lambda, \tilde{w}^k \rangle} \, w^k_i
\end{align*}
and
\begin{align*}
\gamma_i
&= \langle c' , v^i \rangle
= \sum_{k=1}^n c'_k \, v^i_k
= \sum_{k=1}^n c'_k \, w^k_i \, .
\end{align*}
Hence we can write
$F(\lambda) = \sum_{k=1}^n c^*_k \, \e^{\langle \lambda, \tilde{w}^k \rangle} \, w^k$
and $\gamma = \sum_{k=1}^n c'_k \, w^k$.
\begin{definition} \label{F}
Let $V \in \RR^{n \times d}$, $\tilde{V} \in \RR^{n \times \tilde{d}}$ with $n \ge d, \tilde{d}$ have full rank.
We write $V = (v^1,\ldots,v^d) = (w^1,\ldots,w^n)^T$
and $\tilde{V} = (\tilde{v}^1,\ldots,\tilde{v}^{\tilde{d}}) = (\tilde{w}^1,\ldots,\tilde{w}^n)^T$.
Further, let $c^*>0$.
We define
\begin{align*}
F \colon & \RR^{\tilde{d}} \to C^\circ \subseteq \RR^d \\
& \lambda \mapsto \sum_{k=1}^n c^*_k \, \e^{\langle \lambda, \tilde{w}^k \rangle} \, w^k \, ,
\end{align*}
where
\begin{equation*}
C^\circ = \{ \sum_{k=1}^n c'_k \, w^k \in \RR^d \st c' \in \RR^n_> \} \, .
\end{equation*}
\end{definition}

This definition is more transparent than the equivalent one given above.
It becomes clear that the set $C^\circ$ is the interior of the polyhedral cone generated by the vectors $(w^1,\ldots,w^n)$.
The map $F$ itself (in case $V = \tilde{V}$)
appears in toric geometry \cite{Fulton1993}, where it is related to moment maps,
and in statistics \cite{PachterSturmfels2005}, where it is related to exponential families.
There is a useful result \cite{Fulton1993},
which guarantees injectivity and surjectivity of $F$ in case $V = \tilde{V}$.
\begin{proposition} \label{fulton}
Let $V$, $\tilde{V}$, and $F$ be as in Definition \ref{F}.
If $V=\tilde{V}$,
then $F$ is a real analytic isomorphism of $\RR^d$ onto $C^\circ$ for all $c^*>0$.
\end{proposition}

This is a variant of Birch's Theorem \cite{PachterSturmfels2005,Sturmfels2002,CraciunDickensteinShiuSturmfels2009};
it implies Theorem \ref{thm}.
We will build on this result when we study the surjectivity of $F$,
but first we deal with its injectivity in case $V \not= \tilde{V}$.


\subsection{Injectivity of $F$}

In the context of  multiple equilibria in mass action systems \cite{CraciunFeinberg2005}
and geometric modeling \cite{CraciunGarcia-PuenteSottile2010},
it was shown that
the map $F$ (in case $d=\tilde{d}$) is injective for all $c^*$
if and only if
$F$ is a local isomorphism for all $c^*$. 
We give an alternative proof of this result
and extend it to the case $d\not=\tilde{d}$,
where we use the sign vectors of the spaces $\im(V)$ and $\im(\tilde{V})$.

\begin{theorem} \label{inj}
Let $V$, $\tilde{V}$, and $F$ be as in Definition \ref{F}.
Then, the following statements are equivalent:
\begin{remunerate}
\item
$F$ is injective for all $c^*>0$.
\item
$F$ is an immersion for all $c^*>0$. ($\DD{F}{\lambda}$ is injective for all $\lambda$ and $c^*>0$.)
\item
$\sigma(\im(V)^\bot) \cap \sigma(\im(\tilde{V})) = \{0\}$.
\end{remunerate}
\end{theorem}

\vspace{1ex}
\begin{proof}
We use $F$ in the form of Eqn.\ \eqref{Fequ}.

$(1 \Leftrightarrow 3)$: By Proposition \ref{unique}. \\
Using $S^\bot = \im(V)$ and $\tilde{S}^\bot = \im(\tilde{V})$,
the injectivity of $F$ for all $c^*$ is equivalent to the existence of at most one element in
$(c'+S)_\geq \cap \{ c^* \mal \e^{\tilde{v}} \st \tilde{v} \in \tilde{S}^\bot \}$
for all $c'$ and $c^*$.

$(\neg 2\Rightarrow\neg 3)$:
Suppose that $\DD{F}{\lambda}$ is not injective
(for a certain $c^*$ and a certain $\lambda$),
i.e.\ there exists a nonzero $\lambda' \in \RR^{\tilde{d}}$
such that $\DD{F}{\lambda} \lambda' = 0$.
Since
\begin{equation*}
\sum_{j=1}^{\tilde{d}} \DD{F_i}{\lambda_j} \, \lambda_j' =
\sum_{j=1}^{\tilde{d}}
\langle c^* \mal \expp{\sum_{k=1}^{\tilde{d}} \lambda_k \tilde{v}^k} \mal \tilde{v}^j , v^i \rangle \lambda_j' =
\langle \underbrace{c^* \mal \expp{\sum_{k=1}^{\tilde{d}} \lambda_k \tilde{v}^k}}_{c} \mal
\underbrace{\sum_{j=1}^{\tilde{d}} \lambda_j' \, \tilde{v}^j}_{\tilde{v}'} , v^i \rangle \, ,
\end{equation*}
this is equivalent to the existence of $c>0$ and $\tilde{v}' \in \im(\tilde{V})$
such that $\langle c \mal \tilde{v}', v^i \rangle = 0$ for $i=1,\ldots,d$,
which in turn is equivalent to $c \mal \tilde{v}' \in \im(V)^\bot$.
Clearly
$\sigma( c \mal \tilde{v}' )
= \sigma ( \tilde{v}' )$
and hence
$\sigma(\im(V)^\bot) \cap \sigma(\im(\tilde{V})) \neq \{0\}$.

$(\neg 3\Rightarrow\neg 2)$:
Suppose that $0 \not= \tau \in \sigma(\im(V)^\bot) \cap \sigma(\im(\tilde{V}))$.
Then, there exist $u \in \im(V)^\bot$ and $\tilde{v}' \in \im(\tilde{V})$
such that $\sigma(u) = \sigma(\tilde{v}') = \tau$.
Clearly, one can choose $c>0$ such that $u = c \mal \tilde{v}'$
and hence $c \mal \tilde{v}' \in \im(V)^\bot$.
As demonstrated in the previous step,
this is equivalent to the existence of $c^*>0$ and $\lambda, \lambda' \neq 0 \in \RR^{\tilde{d}}$
such that $\DD{F}{\lambda} \lambda' = 0$.
\end{proof}

Finally, we note that
for $d=\tilde{d}$,
Statement 3 in Theorem \ref{inj} is symmetric with respect to $V$ and $\tilde{V}$.
\begin{corollary}
Let $V$, $\tilde{V}$ be as in Definition \ref{F} with $d=\tilde{d}$.
Then,
$\sigma(\im(V)^\bot) \cap \sigma(\im(\tilde{V})) = \{0\}$
if and only if
$\sigma(\im(\tilde{V})^\bot) \cap \sigma(\im(V)) = \{0\}$.
\end{corollary}

\vspace{1ex}
\begin{proof}
Let $F$ be in the form of Eqn.\ \eqref{Fequ} with $d=\tilde{d}$,
and let $\tilde{F}$ be obtained from $F$ by changing the roles of $V$ and $\tilde{V}$,
\begin{align*}
\tilde{F} \colon & \RR^d \to \RR^d \\
& \lambda \mapsto \tilde{F}(\lambda)
\quad \text{with} \quad
(\tilde{F}(\lambda))_i = \langle c^* \mal \expp{\sum_{j=1}^d \lambda_j v^j} , \tilde{v}^i \rangle \, .
\end{align*}
We will show that $\DD{F}{\lambda}$ is injective for all $c^*$
if and only if
$\DD{\tilde{F}}{\lambda}$ is injective for all $c^*$.
Then, by Theorem \ref{inj} we will obtain the desired result.

Suppose that $\DD{\tilde{F}}{\lambda}$
(or equivalently its transpose)
is not injective (for a certain $c^*$ and a certain $\lambda$),
i.e.\ there exists $\lambda' \in \RR^{\tilde{d}}$
such that $(\DD{\tilde{F}}{\lambda})^T \lambda' = 0$.
Since
\begin{align*}
\sum_{j=1}^d \DD{\tilde{F}_j}{\lambda_i} \, \lambda_j' &=
\sum_{j=1}^d
\langle c^* \mal \expp{\sum_{k=1}^d \lambda_k v^k} \mal v^i , \tilde{v}^j \rangle \lambda_j' =
\langle \underbrace{c^* \mal \expp{\sum_{k=1}^d \lambda_k v^k}}_{c} \mal \, v^i ,
\underbrace{\sum_{j=1}^d \lambda_j' \, \tilde{v}^j}_{\tilde{v}'} \rangle
\end{align*}
and $\langle c \mal v^i , \tilde{v}' \rangle = \langle c \mal \tilde{v}' , v^i \rangle$,
this is equivalent to the non-injectivity condition for $\DD{F}{\lambda}$
derived in the proof of Theorem \ref{inj}.
\end{proof}


\subsection{Surjectivity of $F$} \label{subsec:surj}

It is more difficult to derive conditions for the surjectivity of $F$.
Our main result is concerned with sufficient conditions,
however, we start with a discussion of necessary conditions.

Let $C$ and $\tilde{C}$ be the polyhedral cones
generated by the vector configurations $V^T=(w^1, \ldots, w^n)$
and $\tilde{V}^T=(\tilde{w}^1,\ldots,\tilde{w}^n)$, respectively.
Then $C^\circ = \inte(C)$, and analogously we write $\tilde{C}^\circ = \inte(\tilde{C})$.
We note that $C^\circ$ and $\tilde{C}^\circ$ are nonempty
since $V$ and $\tilde{V}$ have full rank.

We write $\sigma(\im(V))_\ge = \sigma(\im(V)) \cap \{0,+\}^n$ for the face lattice of $C$,
see the Appendix,
and analogously $\sigma(\im(\tilde{V}))_\ge = \sigma(\im(\tilde{V})) \cap \{0,+\}^n$ for the face lattice of $\tilde{C}$.
A face $f$ of $C$ is characterized by a sign vector $\tau \in \sigma(\im(V))_\ge$
or equivalently by a supporting hyperplane with normal vector $\lambda \in \RR^d$,
where $\tau_k = 0$ whenever $\langle \lambda , w^k \rangle = 0$ (for $w^k$ lying on $f$)
and $\tau_k = +$ whenever $\langle \lambda , w^k \rangle > 0$.

Now, we can study a necessary condition for surjectivity:
The image of $F$ must contain points arbitrarily close to any point on a face of $C$.
We assume that $C$ is pointed,
more specifically that $(+,\ldots,+)^T \in \sigma(\im(V))$,
and we consider the simplest nontrivial face, namely an extreme ray $e$.
To begin with, we assume that $e$ contains only one generator, say $w^1$;
hence the characteristic sign vector amounts to $\tau=(0,+,\ldots,+)^T$.
If $F$ is surjective, then the cone $\tilde{C}$ must have a corresponding extreme ray $\tilde{e}$
with the same sign vector $\tau$.
Only then there is $\mu \in \RR^{\tilde{d}}$
with $\langle \mu , \tilde{w}^1 \rangle = 0$
and $\langle \mu , \tilde{w}^k \rangle > 0$ for $k=2,\ldots,n$
such that the limit
\[
\lim_{a \to \infty} F(-a \mu + \nu)
= \lim_{a \to \infty}
\sum_{k=1}^n c^*_k \, \e^{-a \langle \mu , \tilde{w}^k \rangle+ \langle \nu , \tilde{w}^k \rangle} \, w^k
= c^*_1 \, \e^{\langle \nu , \tilde{w}^1 \rangle} \, w^1
\]
can be placed arbitrarily close to any point on $e$ (by appropriate choice of $\nu \in \RR^{\tilde{d}}$).

If the extreme ray $e$ contains more than one generator,
there may be several corresponding extreme rays $\tilde{e}$.
For a particular $\tilde{e}$ with characteristic sign vector $\tilde{\tau}$,
there is $\mu \in \RR^{\tilde{d}}$
(with $\langle \mu , \tilde{w}^k \rangle = 0$ if $\tilde{\tau}_k = 0$
and $\langle \mu , \tilde{w}^k \rangle > 0$ if $\tilde{\tau}_k = +$)
such that $\lim_{a \to \infty} F(-a \mu + \nu)$ lies on $e$.
We note that if $\tilde{e}$ contains $\tilde{w}^k$,
then $e$ must contain $w^k$;
otherwise the limit does not lie on $e$.
This condition on the corresponding extreme rays $\tilde{e}$ and $e$
can be expressed by their characteristic sign vectors $\tilde{\tau}$ and $\tau$,
namely as $\tilde{\tau} \ge \tau$.
For higher-dimensional faces,
similar (but more complicated) conditions can be formulated.

For the proof of the following surjectivity result, we will employ Degree Theory.
In particular, we use two properties of the Brouwer degree $\bd(f,D,y)$
of a continuous function $f \colon \bar{D} \to \RR^d$
defined on the closure of an open and bounded subset $D \subset \RR^d$ (with boundary $\partial D$)
at a value $y \not\in f(\partial D)$:
(i) the degree is invariant under homotopy,
and (ii) if the degree is nonzero, there exists $x$ such that $y=f(x)$,
see \cite{Lloyd1978} or \cite{FonsecaGangbo1995}.
\begin{theorem} \label{surj}
Let $V$, $\tilde{V}$, and $F$ be as in Definition \ref{F}.
If there exists a lattice isomorphism
$\Phi \colon \sigma(\im(\tilde{V}))_\ge \to \sigma(\im(V))_\ge$
with
$\tilde{\tau} \ge \Phi(\tilde{\tau})$
and $(+,\ldots,+)^T \in \sigma(\im(V))$,
then $F$ is surjective for all $c^* > 0$.
\end{theorem}

\vspace{1ex}
\begin{proof}
In order to use the Brouwer degree,
we require a map on a closed and bounded set.
To this end, we define a map $G$ equivalent to $F$
from the interior of $\tilde{C}$ to the interior of $C$
and extend $G$ to the boundaries such that it maps faces to faces.
Then, we cut the pointed cones such that we obtain polytopes $\tilde{P}$ and $P$.
Finally, we define a homotopy between the map $G$
and a homeomorphism between the polytopes guaranteed by the face lattice isomorphism.
As a consequence,
every point in the interior of $P$ has nonzero Brouwer degree and hence is in the image of $G$.
Since the cut of the cone $C$ can be placed at arbitrary distance from the origin,
this holds for every point in the interior of $C$.

Since $(+,\ldots,+)^T \in \sigma(\im(V))$,
the face lattice isomorphism implies $(+,\ldots,+)^T \in \sigma(\im(\tilde{V}))$,
and hence the cones $C$ and $\tilde{C}$ are pointed.
We start by choosing a minimal set of generators for $\tilde{C}$,
which (after reordering) we assume to be $(\tilde{w}^1,\ldots,\tilde{w}^{n_E})$,
where $n_E$ is the number of extreme rays of $\tilde{C}$.
We define an auxiliary map,
\begin{align*}
\tilde{F} \colon & \RR^{\tilde{d}} \to \tilde{C}^\circ \\
                 & \lambda \mapsto \sum_{k=1}^{n_E} \tilde{c}^*_k \,
                 \e^{\langle \lambda , \tilde{w}^k \rangle} \, \tilde{w}^k \, ,
\end{align*}
which is a real analytic isomorphism by Proposition \ref{fulton},
and a composed map,
\begin{align*}
G^\circ \colon & \tilde{C}^\circ \to C^\circ \\
               & x \mapsto F(\tilde{F}^{-1}(x)) \, ,
\end{align*}
which is surjective whenever $F$ is surjective.

Since $G^\circ$ is defined only on $\tilde{C}^\circ$,
we want to extend it continuously to the boundary $\partial \tilde{C}$,
i.e.\ to the faces of the cone.
Let $\tilde{f}$ be a face of $\tilde{C}$.
It contains a subset of the minimal set of generators for $\tilde{C}$,
which (after reordering) we assume to be $(\tilde{w}^1,\ldots,\tilde{w}^{n_{\min}})$.
There may be additional generators on $\tilde{f}$,
which we assume to be $(\tilde{w}^{n_E+1},\ldots,\tilde{w}^{n_E+n_{\add}})$,
where $n_{\min} + n_{\add}$
is the total number of generators on $\tilde{f}$.

Now,
let $(x^i)_{i \in \NN}$ be a sequence with $x^i \in \tilde{C}^\circ$ and $\lim_{i \to \infty} x^i \in \tilde{f}$.
Via the isomorphism $\tilde{F}$,
there is a corresponding sequence $(\lambda^i)_{i \in \NN}$ with $\lambda^i \in \RR^{\tilde{d}}$.
From
\begin{equation*}
\lim_{i \to \infty} x^i = \sum_{k = 1}^{n_{\min}} \tilde{c}^*_k
                          \lim_{i \to \infty} \e^{\langle \lambda^i , \tilde{w}^k \rangle} \, \tilde{w}^k +
                          \sum_{k = n_{\min}+1}^{n_E} \tilde{c}^*_k
                          \lim_{i \to \infty} \e^{\langle \lambda^i , \tilde{w}^k \rangle} \, \tilde{w}^k ,
\end{equation*}
we conclude that
$\lim_{i \to \infty} \e^{\langle \lambda^i , \tilde{w}^k \rangle} \ge 0$
for $k = 1,\ldots,n_{\min}$
and
$\lim_{i \to \infty} \e^{\langle \lambda^i , \tilde{w}^k \rangle} = 0$
for $k = n_{\min}+1,\ldots,n_E$.
Additional generators $\tilde{w}^k$ on $\tilde{f}$ can be written as nonnegative linear combinations
of the minimal generators $(\tilde{w}^1,\ldots,\tilde{w}^{n_{\min}})$
and hence%
\footnote{
By using $\e^{\langle \lambda^i , \sum_{k=1}^{n_{\min}} a_k \tilde{w}^k \rangle}
= \prod_{k=1}^{n_{\min}} \left( \e^{\langle \lambda^i , \tilde{w}^k \rangle} \right)^{a_k}$.
}
we obtain $\lim_{i \to \infty} \e^{\langle \lambda^i , \tilde{w}^k \rangle} \ge 0$.
Generators $\tilde{w}^k$ not on $\tilde{f}$ can be written as nonnegative linear combinations
containing at least one of the remaining minimal generators $(\tilde{w}^{n_{\min}+1},\ldots,\tilde{w}^{n_E})$
and hence%
\footnote{
By using $\e^{\langle \lambda^i , \sum_{k=1}^{n_E} a_k \tilde{w}^k \rangle}
= \prod_{k=1}^{n_{\min}}     \left( \e^{\langle \lambda^i , \tilde{w}^k \rangle} \right)^{a_k}
  \prod_{k=n_{\min}+1}^{n_E} \left( \e^{\langle \lambda^i , \tilde{w}^k \rangle} \right)^{a_k}$.
}
we obtain $\lim_{i \to \infty} \e^{\langle \lambda^i , \tilde{w}^k \rangle} = 0$.
As a consequence, the image of the sequence converges and
\begin{equation*}
\lim_{i \to \infty} G^\circ(x^i) = \sum_{k = 1}^{n_{\min}} c^*_k
                                   \lim_{i \to \infty} \e^{\langle \lambda^i , \tilde{w}^k \rangle} \, w^k +
                                   \sum_{k = n_E + 1}^{n_E + n_{\add}} c^*_k
                                   \lim_{i \to \infty} \e^{\langle \lambda^i , \tilde{w}^k \rangle} \, w^k \, .
\end{equation*}
The isomorphism $\Phi$ (between the face lattices of $\tilde{C}$ and $C$) with $\tilde{\tau} \ge \Phi(\tilde{\tau})$
implies that there is a face $f$ of $C$ with $w^k \in f$ if $\tilde{w}^k \in \tilde{f}$.
That is, $w^1,\ldots,w^{n_{\min}} \in f$ as well as $w^{n_E + 1},\ldots,w^{n_E + n_{\add}} \in f$
and hence $\lim_{i \to \infty} G^\circ(x^i) \in f$.

In other words,
there is a continuous extension of $G^\circ$ to the face $\tilde{f}$,
which maps $\tilde{f}$ to the corresponding face $f$.
We set
$G := G^\circ$ on $\tilde{C}^\circ$ and
$G(x) := \lim_{i \to \infty} G^\circ(x^i)$
for any sequence $(x^i)_{i \in \NN}$ with $x^i \in \tilde{C}^\circ$ and $\lim_{i \to \infty} x^i = x \in \tilde{f}$.
Since this can be done for all faces of $\tilde{C}$,
there is a map $G \colon \tilde{C} \to C$ which extends $G^\circ$ continuously to $\partial \tilde{C}$
and maps faces to faces.

Due to the face lattice isomorphism,
a minimal set of generators for $C$ is given by $(w^1,\ldots,w^{n_E})$.
The isomorphism further implies $d=\tilde{d}$.

Since $C$ is a pointed cone,
we can choose a $(d-1)$-dimensional subspace of $\RR^d$ such that $C$ lies on one side of the subspace.
We cut $C$ with a hyperplane parallel to the subspace
and obtain a polytope $P$ (lying on one side of the hyperplane).
In particular, we intersect the extreme rays of $C$ with the hyperplane:
the intersection of the extreme ray $e^k$ (generated by $w^k$) is located at $\alpha_k w^k$ with $\alpha_k > 0$.

Analogously,
we cut $\tilde{C}$ with a hyperplane
and obtain a polytope $\tilde{P}$.
The intersection of the extreme ray $\tilde{e}^k$ (generated by $\tilde{w}^k$)
is located at $\tilde{\alpha}_k \tilde{w}^k$ with $\tilde{\alpha}_k > 0$.

From now on, we restrict the map $G$ to $\tilde{P}$
and choose $\tilde{c}^*$ such that $G$ maps corners of $\tilde{P}$ to corresponding corners of $P$.
For example, the corner $\tilde{\alpha}_1 \tilde{w}^1$ on $\tilde{e}^1$
corresponds (by $\tilde{F}^{-1}$) to
the sequence $(\lambda^i)_{i \in \NN}$ with $\lambda^i \in \RR^n$,
$\lim_{i \to \infty} \tilde{c}^*_1 \, \e^{\langle \lambda^i , \tilde{w}^1 \rangle} = \tilde{\alpha}_1$,
and $\lim_{i \to \infty} \e^{\langle \lambda^i , \tilde{w}^k \rangle} = 0$ for $k = 2, \ldots, n_E$.
In turn, $(\lambda^i)_{i \in \NN}$ corresponds (by $F$) to the corner $\alpha_1 w^1$ on $e^1$:
\[
\lim_{i \to \infty} \left( c^*_1 \, \e^{\langle \lambda^i , \tilde{w}^1 \rangle} \, w^1
+ \sum_{k=n_E+1}^{n_E+n_{\add}} c^*_k \, \e^{\langle \lambda^i , \tilde{w}^k \rangle} \, w^k \right) = \alpha_1 w^1 \, .
\]
Here, we have assumed that in addition to $\tilde{w}^1$ there are additional generators $\tilde{w}^k$
(with $k=n_E+1,\ldots,n_E+n_{\add}$) on $\tilde{e}^1$ with corresponding generators $w^k$ on $e^1$.
If we write $\tilde{w}^k = \tilde{\beta}_k \tilde{w}^1$, $w^k = \beta_k w^1$, and
$x = \lim_{i \to \infty} \e^{\langle \lambda^i , \tilde{w}^1 \rangle}$,
we can determine $\tilde{c}^*_1$ from
\[
\tilde{c}^*_1 \, x = \tilde{\alpha}_1
\quad \text{with} \quad
c^*_1 \, x + \sum_{k=n_E+1}^{n_E+n_{\add}} c^*_k \, x^{\tilde{\beta}_k} \beta_k = \alpha_1 \, .
\]
If we choose $\tilde{c}^*_k$ accordingly for each extreme ray $\tilde{e}_k$,
then $G$ maps ``side-edges'' of $\tilde{P}$ to corresponding side-edges of $P$.
The image of other faces of $\tilde{P}$
need not coincide with the corresponding faces of $P$.
(However, due to the face lattice isomorphism,
the image of a ``side-face'' of $\tilde{P}$\, lying on a face of $\tilde{C}$,
lies in the corresponding face of $C$.)
In particular%
\footnote{
A point on the cut-face of $\tilde{P}$ is a convex combination
of the ``corners'' $\tilde{\alpha}_k \tilde{w}^k$, $k=1,\ldots,n_E$.
By $\tilde{F}^{-1}$ it corresponds to some $\lambda \in \RR^n$,
which by $F$ corresponds to a point on the cut-face of $P$,
that is, a convex combination (with the same coefficients) of the corners $\alpha_k w^k$, $k=1,\ldots,n_E$,
{\em plus} a positive linear combination of the additional generators $w^k$, $k=n_E+1,\ldots,n$.
},
the image of the ``cut-face'' of $\tilde{P}$ (arising from the cut with the hyperplane)
may lie outside the cut-face of $P$.

The isomorphism between the face lattices of $\tilde{C}$ and $C$
has another important consequence.
It guarantees the existence of a piecewise linear homeomorphism $G' \colon \tilde{P} \to P$,
which restricts to homeomorphisms between corresponding faces of $\tilde{P}$ and $P$, see the Appendix.
We note that $G'$ has nonzero Brouwer degree on $P^\circ=\inte(P)$
and define a homotopy between $G$ (restricted to $\tilde{P}$) and $G'$,
\begin{align*}
H \colon & \tilde{P} \times [0,1] \to C \subset \RR^d \\
         & (x, t) \mapsto t \, G(x) + (1-t) \, G'(x) \, .
\end{align*}
(The homotopy $H$ maps to $C$, since both $G$ and $G'$ map to $C$ and $C$ is convex.)

Now, let $y \in P^\circ$.
Below we will show that $y \not\in H(\partial \tilde{P},t)$ for all $t \in [0,1]$.
Writing $\tilde{P}^\circ=\inte(\tilde{P})$,
we conclude that $\bd(G,\tilde{P}^\circ,y) = \bd(G',\tilde{P}^\circ,y) \not= 0$
(by the homotopy invariance of the Brouwer degree)
and that there exists $x \in \tilde{P}^\circ$ with $G(x)=y$ (by the existence property of the Brouwer degree).
In other words, the image of $G$ restricted to $\tilde{P}^\circ$ contains $P^\circ$.
Since the cut of the cone $C$ can be placed at arbitrary distance from the origin,
$G^\circ \colon \tilde{C}^\circ \to C^\circ$ and hence $F \colon \RR^d \to C^\circ$ are surjective.

It remains to show that $y \not\in H(\partial \tilde{P},t)$ for all $t \in [0,1]$:
For side-faces $\tilde{f} \subset \partial \tilde{P}$,
one has $H(\tilde{f},t) \subset \partial C$ for all $t \in [0,1]$
(since $G$ and $G'$ map side-faces to side-faces),
whereas for the cut-face $\tilde{f} \subset \partial \tilde{P}$,
one either has $H(\tilde{f},t) \subset \partial P$ for all $t \in [0,1]$
(whenever $G$ maps one cut-face to the other)
or $H(\inte(\tilde{f}),t) \cap P = \emptyset$ for all $t \in [0,1]$
(whenever $G$ maps the cut-face out of $P$).
In each case, one obtains $H(\partial \tilde{P},t) \cap P^\circ = \emptyset$ for all $t \in [0,1]$.
\end{proof}

We think that the technical condition $(+,\ldots,+)^T \in \sigma(\im(V))$ in Theorem \ref{surj},
which requires the cone $C$ to be pointed,
is not necessary, and a similar result can be obtained for arbitrary cones.
However, at the moment we do not have a complete proof for such a theorem.

\subsection{Main results}

The previous two theorems concerned with injectivity and surjectivity of $F$
allow the following generalization of Proposition \ref{fulton} (Birch's Theorem).
\begin{proposition} \label{genfulton}
Let $V$, $\tilde{V}$, and $F$ be as in Definition \ref{F}.
If $\sigma(\im(V))=\sigma(\im(\tilde{V}))$
and $(+,\ldots,+)^T \in \sigma(\im(V))$,
then $F$ is a real analytic isomorphism of $\RR^d$ onto $C^\circ$ for all $c^* > 0$.
\end{proposition}

\begin{proof}
From $\sigma(\im(V))=\sigma(\im(\tilde{V}))$ it follows that $d=\tilde{d}$
and with Eqn.~\eqref{eq:orthsubspace}
that $\sigma(\im(V)^\bot) \cap \sigma(\im(\tilde{V})) = \{ 0 \}$.
Hence, $F$ is injective and a local isomorphism by Theorem \ref{inj}.
Moreover, with $\Phi$ being the identity,
$F$ is surjective by Theorem \ref{surj}.
\end{proof}

Note that the condition $\sigma(\im(V))=\sigma(\im(\tilde{V}))$ in the previous proposition can be tested algorithmically using chirotopes,
see the Appendix.
We can now formulate a result analogous to Theorem \ref{thm}
in the case of generalized mass action kinetics.
\begin{theorem} \label{genthm}
Let $(\mathscr{S},\mathscr{C},\tilde{\mathscr{C}},\mathscr{R},k)$ be a generalized mass action system with
nonempty set $\tilde{Z}$ of complex balancing equilibria,
stoi\-chio\-metric subspace $S$,
and \kinsub subspace $\tilde{S}$.
If $\sigma(S) = \sigma(\tilde{S})$
and $(+,\ldots,+)^T \in \sigma(S^\bot)$,
then $\tilde{Z}$ meets every stoichiometric compatibility class in exactly one point.
\end{theorem}

\begin{proof}
Suppose $\tilde{Z} \neq \emptyset$.
As discussed at the beginning of Subsection \ref{sec:F},
uniqueness and existence of a complex balancing equilibrium in every stoichiometric compatibility class
correspond to injectivity and surjectivity of the map $F$ as given in Definition \ref{F},
where $V$ and $\tilde{V}$ are bases for $S^\bot$ and $\tilde{S}^\bot$, respectively.
By Eqn.~\eqref{eq:orthsubspace}, $\sigma(S) = \sigma(\tilde{S})$
is equivalent to $\sigma(\im(V))=\sigma(\im(\tilde{V}))$,
and obviously $(+,\ldots,+)^T \in \sigma(S^\bot)$
is equivalent to $(+,\ldots,+)^T \in \sigma(\im(V))$
such that $F$ is injective and surjective by Proposition \ref{genfulton}.
\end{proof}

In the terminology of CRNT, 
a chemical reaction network is {\em conservative} if $S^\bot \cap \RR^\mathscr{S}_> \neq \emptyset$,
i.e.\ if there is a ``vector of molecular weights'',
relative to which all reactions are mass conserving.
Note that the condition $(+,\ldots,+)^T \in \sigma(S^\bot)$ in Theorem \ref{genthm}
means that the underlying chemical reaction network is conservative. 


\section{Examples} \label{sec:exa}

We discuss two examples of generalized mass action systems.
First, we continue the example of the generalized chemical reaction network
introduced in Section~\ref{sec:gmas},
\begin{equation}
\begin{array}{ccc}
A + B & \rightleftharpoons & C \\
\vdots && \vdots \\
a A + b B && c \hspace{1pt} C
\end{array}
\end{equation}
with $a,b,c \in \RR_>$.
The \kin complexes $a A + b B$ and $c \hspace{1pt} C$
(associated with the complexes $A + B$ and $C$)
determine the exponents in the rate functions
$k_{A + B \to C} [A]^a [B]^b$ and $k_{C \to A + B} [C]^c$.

The network is (weakly) reversible
and has 2 complexes and 1 linkage class.
The stoichiometric and \kinsub subspace amount to
$S = \spann \{(-1,-1,1)^T\}$ and $\tilde{S} = \spann \{(-\nu_A,-\nu_B,\nu_C)^T\}$
with dimensions $d=\tilde{d}=1$.
By Proposition \ref{prodef'}, $\delta=\tilde{\delta}=2-1-1=0$,
and by Proposition \ref{prodefzero'}, $\tilde{Z} \neq \emptyset$.
Further, the sign vectors of $S$ and $\tilde{S}$ coincide,
i.e.\ $\sigma(S) = \sigma(\tilde{S})$,
and $(1,1,2)^T \in S^\bot$,
which implies
$(+,+,+)^T \in \sigma(S^\bot)$.
Hence, by Theorem \ref{genthm},
every stoichiometric compatibility class
contains exactly one complex balancing equilibrium.

In the rest of this section,
we study an autocatalytic mechanism (for the overall reaction $A + B \rightleftharpoons C$)
endowed with generalized mass action kinetics:
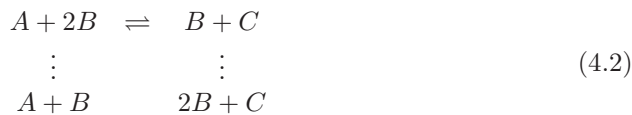
\begin{equation} \label{ex2}
\begin{array}{ccc}
A + 2 B & \rightleftharpoons & B + C \\
\vdots && \vdots \\
A + B && 2 B + C
\end{array}
\end{equation}

The \kin complexes $A + B$ and $2 B + C$
(associated with the complexes $A + 2 B$ and $B + C$)
determine the rate functions
$k_{A + 2 B \to B + C} [A] [B]$ and $k_{B + C \to A + 2 B} [B]^2 [C]$.
The particular kinetics may be unrealistic from a chemical point of view,
however, it will serve to demonstrate how the conditions in Theorem \ref{genthm}
for existence and uniqueness
of a complex balancing equilibrium (in every stoichiometric compatibility class) are violated.

The network is weakly reversible, $\delta=\tilde{\delta}=0$, and hence $\tilde{Z} \neq \emptyset$.
In particular, the stoichiometric and \kinsub subspace amount to
$S = \spann \{(-1,-1,1)^T\}$ and $\tilde{S} = \spann \{(-1,1,1)^T\}$.
For the orthogonal complements $S^\bot$ and $\tilde{S}^\bot$ we choose the bases
\begin{equation*}
V =
\begin{pmatrix}
1 & 0 \\
0 & 1 \\
1 & 1
\end{pmatrix}
\quad \text{and} \quad
\tilde{V} =
\begin{pmatrix}
1 & 1 \\
0 & 1 \\
1 & 0
\end{pmatrix} \, .
\end{equation*}

The cones $C$ and $\tilde{C}$
generated by $V^T=(w^A,w^B,w^C)$
and $\tilde{V}^T=(\tilde{w}^A,\tilde{w}^B,\tilde{w}^C)$
both coincide with $\RR^2_\ge$:

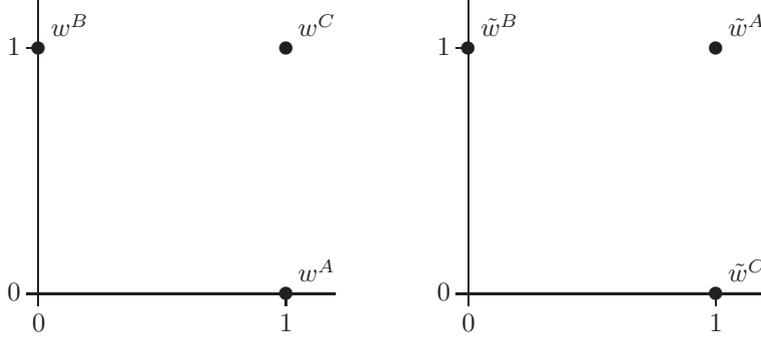
\setlength{\unitlength}{0.25\textwidth}
\hfill
\begin{picture}(1.4, 1.4)(-.2,-.2)
\put(0, 0){\line(0, 1){1.2}}
\put(0, 0){\line(1, 0){1.2}}
\put(1, 0){\circle*{.05}} \put(1.05, 0.05){$w^A$}
\put(0, 1){\circle*{.05}} \put(0.05, 1.05){$w^B$}
\put(1, 1){\circle*{.05}} \put(1.05, 1.05){$w^C$}
\put(0, 0){\line(0, -1){0.05}} \put(1, 0){\line(0, -1){0.05}}
\put(0, 0){\line(-1, 0){0.05}} \put(0, 1){\line(-1, 0){0.05}}
\put(-0.025, -0.15){0}  \put(0.975, -0.15){1}
\put(-0.125, -0.025){0} \put(-0.125, 0.975){1}
\end{picture}
\hfill
\begin{picture}(1.4, 1.4)(-.2,-.2)
\put(0, 0){\line(0, 1){1.2}}
\put(0, 0){\line(1, 0){1.2}}
\put(1, 0){\circle*{.05}} \put(1.05, 0.05){$\tilde{w}^C$}
\put(0, 1){\circle*{.05}} \put(0.05, 1.05){$\tilde{w}^B$}
\put(1, 1){\circle*{.05}} \put(1.05, 1.05){$\tilde{w}^A$}
\put(0, 0){\line(0, -1){0.05}} \put(1, 0){\line(0, -1){0.05}}
\put(0, 0){\line(-1, 0){0.05}} \put(0, 1){\line(-1, 0){0.05}}
\put(-0.025, -0.15){0}  \put(0.975, -0.15){1}
\put(-0.125, -0.025){0} \put(-0.125, 0.975){1}
\end{picture}
\hfill
\phantom{.}

First, we address the question of existence.
We observe that the cone $C$ has an extreme ray generated by $w^A$,
whereas the cone $\tilde{C}$ does not have a corresponding extreme ray generated by $\tilde{w}^A$
(since $\tilde{w}^A$ lies in the interior of $\tilde{C}$).
As a consequence, the map $F$ is not surjective for all $c^*$,
cf.\ the argument at the beginning of Subsection\ \ref{subsec:surj}.
In other words,
there may be a stoichiometric compatibility class
that does not contain a complex balancing equilibrium.

Now, we turn to the question of uniqueness.
In order to employ Proposition \ref{unique} or \ref{notunique},
we determine $\sigma(S) \cap \sigma(\tilde{S}^\bot)$.
The sign vectors of $S$ are $(-,-,+)^T$, its inverse, and $0$,
whereas the sign vectors of $\tilde{S}^\bot$ can be read off from the above figure:
For every hyperplane of $\RR^2$, i.e.\ for every line through $0\in\RR^2$,
we check if $\tilde{w}^A$, $\tilde{w}^B$, and $\tilde{w}^C$
lie on the line or on its negative or positive side.
We obtain
\begin{equation*}
\sigma(S) =
\begin{pmatrix}
- & & 0 \\
- & \ldots & 0 \\
+ & & 0
\end{pmatrix}
\quad \text{and} \quad
\sigma(\tilde{S}^\bot) =
\begin{pmatrix}
+ & + & + & 0 & - & - & & 0 \\
+ & 0 & - & - & - & - & \ldots & 0\\
+ & + & + & + & + & 0 & & 0
\end{pmatrix} \, ,
\end{equation*}
where we use matrix notation for sets of vectors
and where we do not state vectors explicitly that are inverses of others.
We find that $\sigma(S) \cap \sigma(\tilde{S}^\bot)$ contains $(-,-,+)^T$.
Hence,
by Proposition~\ref{notunique},
there exist rate constants $k_{A + 2 B \to B + C}$ and $k_{B + C \to A + 2 B}$
such that some stoichiometric compatibility class
contains more than one complex balancing equilibrium.

Due to the simplicity of the generalized mass action system,
the equilibria of the associated ODE can be determined analytically.
The equilibrium condition amounts to
\begin{equation*}
k_{A + 2 B \to B + C} [A] [B] = k_{B + C \to A + 2 B} [B]^2 [C] \, ,
\end{equation*}
and since $\delta=0$ all equilibria are complex balancing.
By using the conservation relations $[A]+[C]=[A]_0+[C]_0=\Sigma_{AC}$ and $[B]+[C]=[B]_0+[C]_0=\Sigma_{BC}$
and by writing $K = k_{A + 2 B \to B + C} / k_{B + C \to A + 2 B}$,
we obtain a quadratic equation in $[C]$,
which can be solved as
\begin{equation*}
[C] = \frac{K+\Sigma_{BC}}{2} \pm
\sqrt{\left(\frac{K+\Sigma_{BC}}{2}\right)^2
- K \, \Sigma_{AC}} \, .
\end{equation*}
Depending on the equilibrium constant $K$ and the initial values $\Sigma_{AC}$ and $\Sigma_{BC}$
(which determine a stoichiometric compatibility class),
the quadratic equation has 0, 1, or 2 solutions with $[C]>0$.
If additionally $[A]=\Sigma_{AC}-[C]>0$ and $[B]=\Sigma_{BC}-[C]>0$,
then $([A],[B],[C])^T$ is a complex balancing equilibrium.
Obviously, a stoichiometric compatibility class
contains 0, 1, or 2 complex balancing equilibria;
it turns out that each case is realized.


\section{Conclusion}

CRNT establishes intriguing results about the ODEs associated with mass action systems,
in particular about the existence, uniqueness, and stability of equilibria.
For application in molecular biology, however,
one would like to have a framework
that permits rate laws more general than mass action kinetics.

In this paper we show that the suggested notion of generalized mass action systems,
which admits arbitrary nonnegative power-law rate functions,
allows to generalize several results of CRNT.
In particular, Theorem~\ref{genthm} essentially states
that if the sign vectors of the stoichiometric and the \kinsub subspace coincide,
there exists a unique complex balancing equilibrium in every stoichiometric compatibility class.

A natural next step is to study other results of CRNT in the case of generalized mass action kinetics,
most importantly,
to analyze the stability of complex balancing equilibria, which is guaranteed in the classical case.
Further, genuinely biological notions such as the {\em robustness} \cite{BatchelorGoulian2003,ShinarFeinberg2010,Steuer2011}
of chemical reaction networks
can be addressed in a framework with more realistic kinetics.


\Appendix

\section{Sign vectors and face lattices}

In this section,
we outline some facts on the relation between sign vectors of vector spaces
and face lattices of polyhedral cones and polytopes.
For further details we refer to~\cite[Ch.~7]{BachemKern1992} and~\cite[Ch.~2, 6]{Ziegler1995}
and to~\cite{Richter-GebertZiegler1997,BjornerLasSturmfelsWhiteZiegler1999} in the context of oriented matroids.

We obtain the \emph{sign vector} $\sigma(x)\in \{-,0,+\}^n$ of a vector $x \in \RR^n$
by applying the sign function componentwise, and we write
\[
\sigma(S) = \{ \sigma(x) \st x \in S \}
\]
for a subset $S \subseteq \RR^n$.

Two sign vectors $\varsigma,\tau \in \{-,0,+\}^n$ are {\em orthogonal},
if $\varsigma_k \tau_k = 0$ for all $k$
or if there exist $k,l$ with $\varsigma_k \tau_k = -$ and $\varsigma_l \tau_l = +$
(where the product on $\{-,0,+\}$ is defined in the obvious way);
we write $\varsigma \bot \tau$.
Note that $\varsigma \bot \tau$ if and only if there are orthogonal vectors $x,y \in \RR^n$
such that $\sigma(x) = \varsigma$ and $\sigma(y) = \tau$.

The {\em orthogonal complement} $\Sigma^\bot$ of a set $\Sigma \subseteq \{-,0,+\}^n$ is defined by
\[
\Sigma^\bot = \{ \varsigma \in \{-,0,+\}^n \st \varsigma \bot \tau \text{ for all } \tau \in \Sigma \} \, .
\]
The sign vectors of the orthogonal complement of a subspace $S\subseteq \RR^n$ are given by
\begin{equation}
\label{eq:orthsubspace}
\sigma(S^\bot)=\sigma(S)^\bot;
\end{equation}
see for example~\cite[Prop.~6.8.]{Ziegler1995}.

Let $V = (v^1, \ldots , v^d) \in \RR^{n\times d}$ with $n \ge d$ have full rank.
Then $V^T = ( w^1 , \ldots , w^n )$ is called a {\em vector configuration} (of $n$ vectors in $\RR^d$).
With $\lambda \in \RR^d$ and $v = \sum_{j=1}^d \lambda_j v^j \in \im(V)$,
we obtain $v_k = \sum_{j=1}^d \lambda_j v^j_k = \sum_{j=1}^d \lambda_j w^k_j = \langle \lambda , w^k \rangle$.
Hence, $\sigma(v)$ describes the positions of the vectors $w^1, \ldots, w^n$
relative to the hyperplane with normal vector $\lambda$.

The \emph{face lattice} of the cone $C$ generated by $w^1,\ldots,w^n$
can be recovered from the sign vectors of the subspace generated by $v^1,\ldots,v^d$.
It is  the set $\sigma(\im(V)) \cap \{0,+\}^n$ with the partial order induced by the relation $0<+$, which we denote by
\[
\sigma(\im(V))_\ge = \sigma(\im(V)) \cap \{0,+\}^n.
\]
A face $f$ of $C$ is characterized by a supporting hyperplane with normal vector $\lambda \in \RR^d$
such that $\langle \lambda , w^k \rangle = 0$ for generators $w^k$ lying on $f$
and $\langle \lambda , w^k \rangle > 0$ for the remaining $w^k$ (thus lying on the positive side of the hyperplane).

A cone $C$ is called {\em pointed} if $C \cap (-C) = \{ 0 \}$
or equivalently if it has vertex $0$. A cone is pointed if and only if it has an extreme ray, and
every pointed polyhedral cone is the conical hull of its finitely many extreme rays.
Note that if $(+,\ldots,+)^T \in \sigma(\im(V))$, the cone $C$ generated by $V^T$ is pointed.

As for polyhedral cones, the faces of a polytope form a lattice.
Two polytopes are \emph{combinatorially equivalent} if their face lattices are isomorphic.
Combinatorial equivalence corresponds to the existence of a piecewise linear homeomorphism between the polytopes
that restricts to homeomorphisms between faces.

The sign vectors $\sigma(\im(V))$ of the subspace $\im(V)$
can be equivalently characterized by the {\em chirotope} $\chi_{V^T}$ of the point configuration $V^T$,
which is defined as the map
\begin{align*}
\chi_{V^T} \colon & \{1,\ldots,n\}^d \to \{-,0,+\} \\
& (i_1, \ldots, i_d) \mapsto \sign(\det(w^{i_1}, \ldots , w^{i_d})) \, .
\end{align*}
The chirotope records for each $d$-tuple of vectors
if it forms a positively (or negatively) oriented basis of $\RR^d$ or it is not a basis.
It can for example be used to test algorithmically if the sign vectors of two subspaces are equal,
that is, to decide if $\sigma(\im(V)) = \sigma(\im(\tilde{V}))$
for two matrices $V, \tilde{V} \in \RR^{n \times d}$.

\subsection*{Acknowlegdements}

We thank Josef Hofbauer for pointing our attention to Degree Theory
and G\"unter M.~Ziegler for answering our questions on Oriented Matroids.
We also acknowledge fruitful discussions with Fran\c{c}ois Boulier and Fran\c{c}ois Lemaire on an earlier version of the paper
and numerous helpful comments from an anonymous referee.

\end{document}